\documentclass[a4paper]{article}
\usepackage{amssymb,amsthm,amsmath,latexsym,tikz,setspace,enumerate}

\newtheorem{theorem}{Theorem}

\newtheorem{lemma}[theorem]{Lemma}
\newtheorem{corollary}[theorem]{Corollary}

\title{Parity linkage and the Erd\H{o}s-P\'osa property of odd cycles through prescribed vertices in highly connected graphs}

\author{Felix Joos}

\begin{document}

%\onehalfspace

\date{}

\maketitle

%\vspace{-0.5cm}

%\begin{center}
%Institut f\"{u}r Optimierung und Operations Research, 
%Universit\"{a}t Ulm, Ulm, Germany\\
%\texttt{felix.joos@uni-ulm.de}
%\end{center}

\begin{abstract}
We show the following for every sufficiently connected graph $G$, any vertex subset $S$ of $G$, 
and given integer $k$:
there are $k$ disjoint odd cycles in $G$ each containing a vertex of $S$
or there is set $X$ of at most $2k-2$ vertices such that $G-X$ does not contain any odd cycle that contains a vertex of $S$.
We prove this via an extension of Kawarabayashi and Reed's result  about
parity-$k$-linked graphs (Combinatorica 29, 215-225).
From this result it is easy to deduce several other well known results
about the Erd\H{o}s-P\'osa property of odd cycles in highly connected graphs.
This strengthens results due to Thomassen (Combinatorica 21, 321-333), and Rautenbach and Reed (Combinatorica 21, 267-278), respectively.
\end{abstract}

{\small \textbf{Keywords:}  cycles, packing, covering}\\
\indent {\small \textbf{AMS subject classification:}
05C70, % Factorization, matching, partitioning, covering and packing

}

\section{Introduction}
We consider only finite and simple graphs.
A family $\mathcal{F}$ of graphs has the \textit{Erd\H{o}s-P\'osa property} if there is a function $f_\mathcal{F}:\mathbb{N}\rightarrow \mathbb{N}$
such that for every positive integer $k$ and every graph $G$,
the graph $G$ contains $k$ disjoint subgraphs from $\mathcal{F}$
or there is a set $X$ of vertices of $G$ with $|X|<f_\mathcal{F}(k)$ such that $G-X$ contains no subgraph from $\mathcal{F}$.
This notion has been introduced because Erd\H{o}s and P\'osa proved
that the family of cycles has the Erd\H{o}s-P\'osa property \cite{EP62}.
It is one facet of the duality between packing and covering in graphs, which is one of the most fundamental concepts in graph theory.
There is a huge number of results about families of graphs which have the Erd\H{o}s-P\'osa property.
For example, Birmel{\'e}, Bondy, and Reed \cite{BBR07} verified it for the family of cycles of length at least $\ell$ for some integer~$\ell$
and Robertson and Seymour \cite{RS86} showed it for the family of graphs that contain a fixed planar graph as a minor.

In contrast, the family of odd cycles does not have the Erd\H{o}s-P\'osa property.
In particular, there is a sequence of graphs $(G_n)_{n\in \mathbb{N}}$ such that $G_n$ does not contain two disjoint odd cycles,
all odd cycles are of length $\Omega(\sqrt{n})$, 
and every set that intersects all odd cycles has cardinality at least $\Omega(\sqrt{n})$ \cite{Ree99}.

However, Thomassen \cite{Tho01} proved that the family of odd cycles has the Erd\H{o}s-P\'osa property if we restrict ourselves to graphs with high connectivity.
\mbox{Rautenbach} and Reed~\cite{RR01} improved Thomassen's connectivity bound from a double-exponential to linear one.
Later, Kawarabayashi and Reed \cite{KR09} lowered this bound to $24k$, and
Kawarabayashi and Wollan~\cite{KW06} improved this further to $\frac{31}{2}k$.

More than 50 years ago, Dirac \cite{Dir60} showed that
in every $k$-connected graph~$G$, there is a cycle containing any prescribed set of $k$ vertices.
Later, Bondy and Lov{\'a}sz \cite{BL81} extended Dirac's result and proved among other results along this line
that for every $k$-connected non-bipartite graph $G$,
there is an odd cycle containing any prescribed set of $k-1$ vertices.
%	So results about cycles through a prescribed set of vertices can already be found in the early stages of graph theory.

If one asks for many disjoint cycles through a prescribed set $S$ of vertices
it is natural to start with disjoint cycles each containing at least one element of $S$.
We call such cycles  \emph{$S$-cycles}.
Pontecorvi and Wollan \cite{PW12} showed that the class $\mathcal{C}_S$ of $S$-cycles has the Erd\H{o}s-P\'osa property
with $f_{\mathcal{C}_S}(k)=O(k\log k)$, 
which improved the quadratic bound from \cite{KKM11}.
Bruhn et al.~\cite{BJS14} proved that the class $\mathcal{C}_\ell$ of all $S$-cycles of length at least $\ell$ has the Erd\H{o}s-P\'osa property
with $f_{\mathcal{C}_\ell}(k)=O(\ell k \log k)$.
For $S=V(G)$, these results yield the Erd\H{o}s-P\'osa property for cycles and cycles of length at least $\ell$, respectively.

Although, the Erd\H{o}s-P\'osa property does not hold for odd cycles,
it is proved in \cite{KK13} that a half-integral version for the Erd\H{o}s-P\'osa property of odd $S$-cycles holds.
This generalizes a result of Reed \cite{Ree99}, who proved the case $S=V(G)$.

In this paper we continue the study of $S$-cycles by showing the following theorem.
We say a set of vertices $X$ is an \emph{odd cycle cover} and an \emph{odd $S$-cycle cover} of $G$ 
if $G-X$ is bipartite and if $G-X$ does not contain an odd $S$-cycle, respectively.
As mentioned above, the results in \cite{KR09,KW06,RR01}
show that linear connectivity ensures that 
a graph has $k$ vertex disjoint odd cycles or an odd cycle cover of size $2k-2$. 
We show that a sufficiently connected graph 
has $k$ vertex disjoint odd $S$-cycles for any prescribed vertex set $S$ or has an odd $S$-cycle cover of size $2k-2$.
Furthermore, if $S$ has size at least $k$, then in the latter case the graph has also an odd cycle cover of size $3k-3$. 
The bound of $2k-2$ and $3k-3$ is tight for any connectivity, respectively.

\begin{theorem}\label{EPoddScycles}
For any integer $k$, any $50k$-connected graph $G$, and any subset $S$ of vertices of $G$,
at least one of the following statements hold:
\begin{enumerate}[(i)]
	\item $G$ contains $k$ disjoint odd $S$-cycles.
	\item There is a set $X$ with $|X|\leq 2k-2$ such that $G-X$ does not contain an odd $S$-cycle
	and if $|S|\geq k$, there is a set $Y$ with $|Y|\leq 3k-3$ such that $G-X$ is bipartite.
\end{enumerate}
\end{theorem}

Observe that the choice $S=V(G)$ implies the results in~\cite{KR09,KW06,RR01}.
In fact, we prove more detailed results than Theorem \ref{EPoddScycles}.
It is not difficult to see that there are arbitrarily highly connected graphs that contain $k$ disjoint odd ($S$-)cycles and an odd ($S$-)cycle cover of size less than $2k-2$.
In this paper, we present an equivalent condition for $50k$-connected graphs for having $k$ disjoint odd ($S$-)cycles
and deduce the known Erd\H{o}s-P\'osa-type result from this result.

We say a graph is $k$-\emph{linked} if for every set of $k$ pairs of distinct vertices $\{\{s_1,t_1\},\ldots,\{s_k,t_k\}\}$,
there are disjoint paths $P_1,\ldots,P_k$ such that $P_i$ connects $s_i$ and $t_i$.
Moreover,
a graph is \emph{parity-$k$-linked} if it is $k$-linked and we can additionally specify 	 
whether the length of each $P_i$ should be odd or even individually for every $1\leq i\leq k$.

There are several results stating that
if $G$ is $g_1(k)$-connected,
then $G$ is $k$-linked.
The best such result is due to Thomas and Wollan \cite{TW05} who proved that $g_1(k)=10k$ suffices.
They even proved the following stronger result.
\begin{theorem}[\cite{TW05}]\label{k-linkage2}
Every $2k$-connected graph $G$ with at least $5k|V(G)|$ edges is $k$-linked.
\end{theorem}

There are also results of the form
if $G$ is $g_2(k)$-connected and without an odd cycle cover of size $4k-4$,
then $G$ is parity-$k$-linked.
In particular, Kawarabayashi and Reed \cite{KR09} proved the following.
\begin{theorem}[\cite{KR09}]\label{parity-k-linkage}
Every $50k$-connected graph without an odd cycle cover of size $4k-4$ is parity-$k$-linked.
\end{theorem}

The condition of having no small odd cycle cover is necessary and best possible --
there are graphs of arbitrarily high connectivity and with an odd cycle cover of size $4k-4$
that are not parity-$k$-linked.
For example, consider a large complete bipartite graph $G$ with bipartition $(A, B)$
where we add to $A$ the edges of a clique on $2k-1$ vertices
and we add to $B$ the edges of a clique on $2k$ vertices $\{s_1,\ldots,s_k,t_1,\ldots,t_k\}$ minus the perfect matching $\{s_1t_1,\ldots,s_kt_k\}$.

One can apply Theorem \ref{parity-k-linkage} almost directly to deduce that every $50k$-connected graph $G$
without an odd cycle cover of size $4k-4$ has $k$ disjoint odd $S$-cycles for any set $S$ of at least $k$ vertices.
However, the bound on the size of the odd cycle cover is not optimal.
In this paper we prove a stronger version of Theorem~\ref{parity-k-linkage},
reprove the Erd\H{o}s-P\'osa property for odd cycles for $50k$-connected graphs,
and as the main result of this paper,
we prove Theorem \ref{EPoddScycles}.

In addition, we prove several results on the way that may be of independent interest.

The paper is organized as follows. 
In Section 2 we deal with the results concerning the parity-$k$-linkage and
in Section 3 we prove the results about the Erd\H{o}s-P\'osa property for odd $S$-cycles.

\section{Highly parity linked graphs}

In the next theorem we explicitly characterize the obstruction 
for a $50k$-connected graph and a set $\{\{s_1,t_1\},\ldots,\{s_k,t_k\}\}$ of $k$ pairs of distinct vertices
for not having $k$ disjoint $P_1,\ldots,P_k$ paths of prescribed length parity where $P_i$ connects $s_i$ and $t_i$.

Before we state the theorem, we introduce some definitions.
A \emph{partition} $(A,B)$ of $G$ is partition of the vertex set of $G$ into two sets $A$ and $B$.
For a partition $(A, B)$ of $G$, we denote by $G_{A,B}$ the graph $G[A] \cup G[B]$.
A partition $(A,B)$ of $G$ is a \emph{bipartition} if $G_{A,B}$ is edgeless.
A partition $(A,B)$ of $G$ is \emph{nice} 
if there is a minimum odd cycle cover $X$ of $G$
for which $(A\setminus X,B\setminus X)$ is a bipartition of $G-X$ 
such that a vertex of $X$ is in $A$ (respectively, $B$) if it has more neighbors in $B\setminus X$ than in $A\setminus X$
(respectively, more neighbors in $A\setminus X$ than in $B\setminus X$).
We say that a minimum odd cycle cover $X$ \emph{induces} some nice partition $(A,B)$ of $G$ if 
$(A\setminus X,B\setminus X)$ is a bipartition such that a vertex of $X$ is in $A$ (respectively, $B$) if it has more neighbors in $B\setminus X$ than in $A\setminus X$
(respectively, more neighbors in $A\setminus X$ than in $B\setminus X$).
Note that every minimum odd cycle cover induces a nice partition.

Let $(A,B)$ be nice partition of $G$ and let $S=\{\{s_1,t_1\},\ldots,\{s_k,t_k\}\}$ be a set of $k$ pairs distinct vertices.
Let $I\subseteq [k]$ be a set of integers.
A \emph{parity breaking matching} for $(S,I)$ (with respect to the partition $(A,B)$) is a matching $M=\{m_i\}_{i\in I}$ such that $M\subseteq E(G_{A,B})$ 
and $m_i\cap \{s_j,t_j\}=\emptyset$, for $i\in I$ and $i\neq j\in [k]$.
If $I=[k]$, we also say $M$ is a parity breaking matching for $S$.
%For simplicity we sometimes consider $S$ as the set $\{s_1,\ldots,s_k,t_1\ldots,t_k\}$ of $2k$ vertices.

\begin{theorem}\label{parity linked fixed set}
Let $k\in \mathbb{N}$ and let $I\subseteq [k]$.
Let $G$ be a $(26k+24|I|)$-connected graph and let $S=\{\{s_1,t_1\},\ldots,\{s_k,t_k\}\}$ be a set of $k$ pairs of distinct vertices.
If there is a nice partition $(A,B)$ of $G$ with a parity breaking matching $M$ for $(S,I)$,
then $G$ contains $k$ disjoint paths $P_1,\ldots,P_k$ such that $P_i$ connects $s_i$ and $t_i$
and for $i\in I$, we can individually prescribe the parity of the length of $P_i$.
%Exactly one of the following two statements holds.
%\begin{enumerate}
%	\item $G$ contains $k$ disjoint paths $P_1,\ldots,P_k$ of any prescribed parity such that $P_i$ connects $s_i$ and $t_i$.
%	\item For all nice partitions of $G$, there is no parity breaking matching for $S$ of size $k$.
%\end{enumerate}
\end{theorem}

Let us make the following observation.
Suppose $G$ is a graph and $S=\{\{s_1,t_1\},\ldots,\{s_k,t_k\}\}$ is a set of $k$ pairs of distinct vertices.
Suppose $G$ contains $k$ disjoint paths $P_1,\ldots,P_k$ such that $P_i$ connects $s_i$ and $t_i$
and for $i\in I$, we can individually prescribe the parity of the length of $P_i$.
Let $(A,B)$ be any partition of $G$.
Let $P_1,\ldots,P_k$ be disjoint paths
where $P_i$ is a $s_i,t_i$-path and for $i\in I$, we choose the parity of $P_i$ to be even if exactly one vertex of $\{s_i,t_i\}$ belongs to $A$ and odd otherwise.
Thus for $i\in I$, the path $P_i$ contains at least one edge $m_i$ in $E(G_{A,B})$.
Therefore, $\{m_i\}_{i\in I}$ is a parity breaking matching for $(S,I)$.
This leads to the following corollary.

\begin{corollary}\label{cor: nice partition}
Let $k\in \mathbb{N}$ and let $I\subseteq [k]$.
Let $G$ be a $(26k+24|I|)$-connected graph and let $S=\{\{s_1,t_1\},\ldots,\{s_k,t_k\}\}$ be a set of $k$ pairs of distinct vertices.
A nice partition of $G$ has a parity breaking matching for $(S,I)$
if and only if every nice partition has a parity breaking matching for $(S,I)$.
\end{corollary}

There are plenty of consequences of Theorem \ref{parity linked fixed set}.
Firstly, it is easy to see that it implies Corollary~\ref{parity linked}.

\begin{corollary}\label{parity linked}
Let $k\in \mathbb{N}$ and let $G$ be a $50k$-connected graph.
Exactly one of the following two statements holds.
\begin{enumerate}[(i)]
	\item $G$ is $k$-parity linked.
	\item There is a set $S=\{\{s_1,t_1\},\ldots,\{s_k,t_k\}\}$ of $k$ pairs of distinct vertices such that
	for all nice partitions of $G$, there is no parity breaking matching for $S$ of size $k$.
\end{enumerate}
\end{corollary}

Secondly, later we deduce Theorem \ref{parity-k-linkage}.
The third consequence (Theorem \ref{parity linked vertex}) shows that the bound ``$4k-4$'' in Theorem \ref{parity-k-linkage} can be strengthened
to ``$2k-2$'' if $\{s_1,\ldots,s_k,t_1,\ldots,t_k\}$ is an independent set.
Note that both bounds ``$4k-4$'' and ``$2k-2$'' are best possible, respectively.
As a fourth consequence we prove the Erd\H{o}s-P\'osa property for odd $S$-cycles (Theorem \ref{EPoddScycles}) in Section 3.

We say that $G$ is \emph{parity-$k$-linked restricted to independent sets}
if for every independent set of $2k$ vertices $\{s_1,\ldots,s_k,t_1,\ldots,t_k\}$,
there are disjoint paths $P_1,\ldots,P_k$ such that $P_i$ connects $s_i$ and $t_i$
and we can choose whether the length of $P_i$ is odd or even.

\begin{theorem}\label{parity linked vertex}
Let $k\in \mathbb{N}$ and let $G$ be a $50k$-connected graph.
At least one of the following statements holds.
\begin{enumerate}[(i)]
	\item $G$ is parity-$k$-linked restricted to independent sets.
	\item There is a set $X$ of $2k-2$ vertices such that $G-X$ is bipartite.
\end{enumerate}
\end{theorem}

Next, we mention two results needed in the proof of Theorem \ref{parity linked fixed set}.
The first result is basically due to Mader and there is a slightly improved version for triangle-free graphs in~\cite{KR09}.
For a graph $G$, let $\delta(G)$ be the minimum degree of $G$. 

\begin{lemma}[Mader \cite{Mad72}]\label{connected2}
If $G$ is a graph such that $\delta(G)\geq 12k$,
then $G$ contains a $2k$-connected graph $H$
with at least $5k|E(H)|$ edges.
\end{lemma}
Using Theorem \ref{k-linkage2}, 
this implies that a graph $G$ with $\delta(G)\geq 12k$ has a subgraph which is $k$-linked.

Another result which is used in the proof of Theorem \ref{parity linked fixed set} is due to Geelen~et~al..
For a graph $G$ and a set of vertices $Z$,
a \emph{$Z$-path} is a path $P$ such that $V(P)\cap Z$ contains exactly the end vertices of $P$.

\begin{theorem}[Geelen~et~al.~\cite{GGRSV09}]\label{odd path hitting set}
For any set $Z$ of vertices of a graph $G$ and any positive integer $\ell$
at least one of the following statements holds.
\begin{enumerate}[(i)]
	\item There are $\ell$ disjoint odd $Z$-paths.
	\item There is a set $X$ of at most $2\ell-2$ vertices such that $G-X$ contains no odd $Z$-path.
\end{enumerate}
\end{theorem}

We proceed with the proof of Theorem~\ref{parity linked fixed set}.

\begin{proof}[Proof of Theorem \ref{parity linked fixed set}]
%Suppose the first statement holds.
%Let $(A, B)$ be some nice partition of $G$.
%Let $P_1,\ldots,P_k$ be disjoint paths
%where $P_i$ is a $s_i,t_i$-path and we choose the parity of $P_i$ to be even if exactly one vertex of $\{s_i,t_i\}$ belongs to $A$ and odd otherwise.
%Thus $P_i$ contains at least one edge $m_i$ in $E(G_{A,B})$.
%Therefore, $\{m_1,\ldots,m_k\}$ is a parity breaking matching for $S$ of size $k$.

Let $(A,B)$ be a nice partition of $G$ with a parity breaking matching $M=\{m_i\}_{i\in I}$.
If an edge of $M$ covers a vertex of $\{s_i,t_i\}$, let $m_i=x_iy_i$ be this edge and choose $x_i,y_i$ such that $x_i=s_i$ or $y_i=t_i$.
Let $X$ be a minimum odd cycle cover of $G$ that induces the nice partition $(A,B)$.

Suppose first that $|X|<8k$.
By the definition of a nice partition and the fact that $G$ is $(26k+24|I|)$-connected, we know that
every vertex in $A\cap X$ has at least $9k$ neighbors in $B\setminus X$
and every vertex in $A\setminus X$ has at least $18k$ neighbors in $B\setminus X$.
Thus every vertex in $A$ has at least $9k$ neighbors in $B\setminus X$.
Let $T=\bigcup_{i=1}^k\{s_i,t_i\}\cup \bigcup_{i\in I}\{x_i,y_i\}$.
Hence every vertex in $A$ has at least $5k$ neighbors in $A\setminus (T\cup X)$.
The same holds vice versa for the vertices in $B$.

Therefore, we can find a set of at most $4k$ distinct vertices 
$$\bigcup_{i=1}^k\{s_i',t_i'\}\cup \bigcup_{i\in I}\{x_i',y_i'\}\subseteq 
V(G)\setminus (T \cup X)$$
such that $z'$ is a neighbor of $z$ for $z\in T$ (symbolically written)
and
exactly one vertex of the set $\{z,z'\}$ belongs to $A$.

Let $G'=G-(T\cup X)$.
Thus $G'$ is $24k$-connected and bipartite.
In addition, by Theorem \ref{k-linkage2}, we obtain that $G'$ is $2k$-linked.
 
Next, we define the desired disjoint paths $P_1,\ldots,P_k$
such that $P_i$ is a $s_i,t_i$-path in $G$ and for $i\in I$ the length of $P_i$ is of the prescribed parity.
In order to do so, we seek for disjoint paths $P_1',\ldots, P_k',P_1'',\ldots, P_k''$ in $G'$ with disjoint end vertices.
As $G'$ is $2k$-linked such disjoint paths exist. 

For every $i$ we proceed as follows.
If $i\notin I$, then let $P_i'$ be a $s_i',t_i'$-path in $G'$ and 
let $P_i$ be  the conjunction of $s_is_i'$, the path $P_i'$, and $t_i't_i$.
Suppose next that $i\in I$.
If our choice of the parity of length of $P_i$ shall respect the parity naturally given by the sides of the partition $(A,B)$,
then let $P_i'$ be a path connecting $s_i'$ and $t_i'$ in $G'$ and 
let $P_i$ be  the conjunction of $s_is_i'$, the path $P_i'$, and $t_i't_i$.
Otherwise,
let $P_i'$ be a path connecting $s_i'$ and $x_i'$ and $P_i''$ be a path connecting $t_i'$ and $y_i'$.
If $\{s_i,t_i\}\cap \{x_i,y_i\}=\emptyset$, then
let $P_i$ be the conjunction of $s_is_i'$, the path $P_i'$, the path $x_i'x_iy_iy_i'$, the path $P_i''$, and $t_i't_i$.
If $s_i=x_i$ and $t_i\not=y_i$,
then let $P_i$ be the conjunction of $s_iy_iy_i'$, the path $P_i''$, and $t_i't_i$.
If $s_i\not=x_i$ and $t_i=y_i$,
then let $P_i$ be the conjunction of $s_is_i'$, the path $P_i'$, and $x_i'x_it_i$.
Finally, if $s_i=x_i$ and $t_i=y_i$, then let $P_i= s_it_i$.
Thus, there are disjoint paths $P_1,\ldots, P_k$ as desired.

%As mentioned above, because $G'$ is $2k$-linked, we can choose the corresponding paths $P_1',P_1'',\ldots, P_k',P_k''$ in $G'$ to be pairwise disjoint and hence
%also $P_1,\ldots,P_k$ are pairwise disjoint.

It remains to show that if $X$ has size at least $8k$, then the first statement holds.
This part of the proof can basically be found in \cite{KR09}.
However, we change some arguments which leads to a shorter proof.
Let $G'=G-\{s_1,\ldots,s_k,t_1,\ldots,t_k\}$ and let $(A',B')$ be a partition of $G'$ 
such that $|E(G_{A',B'}')|$ is minimized.
Define $G''=G'-E(G_{A',B'}')$.
%induced by $(A, B)$.
Note that $\delta(G'')\geq 12(k+|I|)$ and $G''$ is bipartite.
By Lemma \ref{connected2},
there is a 
$2(k+|I|)$-connected subgraph $H$ of $G''$ with $|E(H)|\geq 5(k+|I|) |V(H)|$.
Moreover, by Theorem~\ref{k-linkage2},
the graph $H$ is $(k+|I|)$-linked.
Let $(A_H,B_H)$ be the (unique) bipartition of $H$ such that $A_H\subseteq A'$.
Observe that $|A_H|\geq 10k$, because $\delta(H)\geq 10k$.

Theorem \ref{odd path hitting set} guarantees a set $Y$ with $|Y|\leq 6k-6$ that intersects 
all odd $A_H$-paths in $G'$ or
$3k$ disjoint odd $A_H$-paths in $G'$.
Suppose that there is a set $Y$ of at most $6k-6$ vertices such that $G'-Y$ contains no odd $A_H$-path.
For a contradiction, we assume that $G'-Y$ is not bipartite.
Thus there is an odd cycle $C$ in $G'-Y$.
Since $G'$ is $24(k+|I|)k$-connected, $G'-Y$ is $2$-connected.
Hence there are two disjoint $A_H$-$C$-paths in $G'$.
Note that the length of these paths could be zero.
Nevertheless, combining these two paths with a suitable part of the cycle $C$ leads to an odd $A_H$-path, which is a contradiction.
This in turn implies that $S \cup Y$ is an odd cycle cover of $G$ of size at most $8k-6$, which is a contradiction to the assumption $|X|\geq 8k$.
Thus Theorem~\ref{odd path hitting set} implies the existence of $3k$ disjoint odd $A_H$-paths.

Let $P$ be one of these $3k$ odd $A_H$-path. 
There is a natural partition of $E(P)$ into $V(H)$-paths.
Because $P$ is an odd $A_H$-path, there is a subpath $P'$ of $P$ such that
$P'$ is an odd $H$-path and both end vertices of $P'$ lie in the same side of the bipartition of $H$
or $P'$ is an even $H$-path and exactly one end vertex of $P'$ lies in $A_H$.
To see this, assume for a contradiction that all subpaths are of odd length if exactly one end vertex lies in $A_H$
and of even length of both end vertices lie in $B_H$.
As there are either zero or two paths with exactly one end vertex in $A_H$,
the path $P$ has even length, which is a contradiction.	

Therefore, there is a set $\mathcal{Q}$ of $3k$ disjoint $H$-paths $Q_1,\ldots, Q_{3k}$ where the length of $Q_i$ is odd if both end vertices lie in the same side of the bipartition of $H$ and even otherwise. 

Since $G$ is $26k$-connected,
there is a set of $2k$ disjoint paths $\mathcal{P}=\{P_1,\ldots,P_{2k}\}$ connecting $\{s_1,\ldots,s_k,t_1,\ldots,t_k\}$ and $H$.
Choose these paths such that they intersect as few as possible paths from $\mathcal{Q}$.
Under this condition choose these paths such that their edge intersection with the paths in $\mathcal{Q}$ is as large as possible.
The latter condition implies that 
if $Q\in \mathcal{Q}$ has nonempty intersection with a path in $\mathcal{P}$ --
let $z', z''$ be the end vertices of $Q$, and
let $P$ be first path that intersects $Q$ seen from the direction of $z'$ --
then $P$ follows the path $Q$ up to $z'$ beside the case that $P$ is the only path intersecting $Q$, and $P$ follows $Q$ to $z''$.
Hence for every $Q\in \mathcal{Q}$ that intersects a path in $\mathcal{P}$,
there is at least one path $P\in \mathcal{P}$ such that there is vertex $z$ that is an end vertex of $P$ and $Q$.
Clearly, a path $P\in \mathcal{P}$ can only share its end vertex with one path in $\mathcal{Q}$.
Therefore, the paths in $\mathcal{P}$ intersect at most $2k$ paths in $\mathcal{Q}$
and hence there is a collection $\mathcal{Q'}=\{Q_1',\ldots,Q_{k}'\}\subseteq\mathcal{Q}$ of $k$ paths 
such that $Q\cap P =\emptyset$ for $P\in \mathcal{P}$ and $Q\in\mathcal{Q}'$.

Since $H$ is $(k+|I|)$-linked,
we can find the desired $k$ disjoint paths of specified parity connecting $s_i$ and $t_i$ for $i\in I$ by using the paths $\mathcal{P}$
and then either directly linking the ends in $H$ of the paths belonging to $s_i$ and $t_i$ (we also do this for all $i\notin I$) or by using the path $Q_i'$ in between.
\end{proof}

For a graph $G$, let a set of vertices $X$ of $G$ be a \emph{vertex cover} of $G$
if every edge is incident to at least one vertex of $X$.
Let the \emph{vertex cover number} $\tau(G)$ of $G$ be the least number $k$ such that $G$ has a vertex cover $X$ with $|X|=k$.
Since a vertex cover has to contain at least one vertex of every edge in a matching $M$,
we have on the one hand $|M|\leq \tau(G)$ for every matching $M$ in $G$.
On the other hand, we observe the following.
\begin{equation}\label{vertex cover matching}
\begin{minipage}[c]{0.8\textwidth}\em
If $M$ is a maximal matching of $G$,
then the vertices covered by $M$ form a vertex cover of $G$ and hence $\tau(G)\leq 2|M|$.
\end{minipage}\ignorespacesafterend 
\end{equation}
Trivially, $\tau(G-v)\geq \tau(G)-1$ for every graph $G$ and $v\in V(G)$,
since every vertex cover of $G-v$ together with $\{v\}$ is a vertex cover of $G$.

A graph $G$ is \emph{$\tau$-critical} if $\tau(G-e)< \tau(G)$ and $\tau(G-v)<\tau(G)$ for every edge $e \in E(G)$ and every vertex $v\in V(G)$.
A result of Erd\H{o}s and Gallai \cite{EG61} says that if $G$ is $\tau$-critical, then $\tau(G)\geq |V(G)|/2$.

We make another observation.
\begin{equation}\label{vertex cover occ}
\begin{minipage}[c]{0.8\textwidth}\em
If $G$ is a graph and $(A,B)$ is a nice partition of $G$ induced by the minimum odd cycle cover $X$,
then $|X|=\tau(G_{A,B})$.
\end{minipage}\ignorespacesafterend 
\end{equation}
This can be seen as follows. 
As $G-X$ is a bipartite graph with bipartition $(A\setminus X,B\setminus X)$,
the set $X$ is a vertex cover of $G_{A,B}$.
Thus $|X|\geq \tau (G_{A,B})$.
Suppose $Y$ is a vertex cover of $G_{A,B}$,
then $G-Y$ is a bipartite graph and hence $|X|\leq \tau (G_{A,B})$.

Having these definitions in mind we reprove Kawabarayashi's and Reed's result and directly afterwards Theorem~\ref{parity linked vertex}.

\begin{proof}[Proof of Theorem \ref{parity-k-linkage}]
Suppose that $X$ is a minimum odd cycle cover and $|X|\geq 4k-3$.
We show by induction on $k$ that $G$ contains a parity breaking matching for $S$.
Let $(A, B)$ be a nice partition induced by $X$. %and let $H=G[A]\cup G[B]$.
By \eqref{vertex cover occ}, $X$ is a minimum vertex cover of $G_{A,B}$.

We show that for every graph $H$, 
any set $\{\{x_1,y_1\},\ldots,\{x_k,y_k\}\}$ of $k$ pairs of distinct vertices in $H$ 
such that $\tau(H)\geq 4k-3$,
there is a matching $M=\{m_1,\ldots,m_k\}$ in $H$ such that $m_i\cap \{x_j,y_j\}=\emptyset$ for $i\neq j$.

Let $Y$ be a minimal vertex cover of $H$ and we prove the statement by induction on $k$.

Suppose $k=1$.
Since $\tau(H)\geq 1$, the graph $H$ must contain an edge $e$ and $\{e\}$ is the desired matching.
Hence we may assume that $k\geq 2$.
Because $2k<4k-3$, there is a vertex $r$ such that $r\in Y\setminus \{x_1,\ldots,x_k,y_1,\ldots,y_k\}$.
Since $|Y|=\tau(H)$, 
the vertex $r$ has a neighbor $r'$.
Note that either $r'\notin \{x_1,\ldots,x_k,y_1,\ldots,y_k\}$
or we may assume by symmetry that $r'=x_k$.
Observe that $\tau(H-\{x_k,y_k,r,r'\})\geq 4(k-1)-3$ and 
combining $rr'$ and the induction hypothesis for $H-\{s_k,t_k,r,r'\}$,
we conclude that $H$ contains a matching $M$ as desired.

Thus with $G_{A,B}$ playing the role of $H$ and $s_i,t_i$ playing the role of $x_i,y_i$,
the graph $G$ contains a parity breaking matching for $\{\{s_1,t_1\},\ldots,\{s_k,t_k\}\}$  and applying Theorem \ref{parity linked fixed set} completes the proof.
\end{proof}

\begin{proof}[Proof of Theorem \ref{parity linked vertex}]
Suppose that $X$ is a minimum odd cycle cover of $G$ and $|X|\geq 2k-1$.
Fix some set $S=\{\{s_1,t_1\},\ldots,\{s_k,t_k\}\}$ of $k$ pairs of distinct vertices which are independent in $G$.
Let $(A, B)$ be a nice partition induced by $X$.
By Theorem~\ref{parity linked fixed set} it suffices to show that $G$ contains a parity breaking matching for $S$.
Note that $X$ is a minimum vertex cover of $G_{A,B}$
and hence $\tau(G_{A,B})\geq 2k-1$.

We prove that following statement, which clearly completes the proof of Theorem~\ref{parity linked vertex}.
For every graph $H$ and every set
$\{\{x_1,y_1\},\ldots,\{x_k,y_k\}\}$ of $k$ pairs of vertices
such that $Z=\{x_1,\ldots,x_k,y_1,\ldots,y_k\}$ is an independent set of size $2k$
and $\tau(H)\geq 2k-1$,
there is a matching $M=\{m_1,\ldots,m_k\}$ in $H$ such that $m_i\cap \{x_j,y_j\}=\emptyset$ for $i\neq j$.

We proceed by induction on $k$.
If $k=1$, then $H$ contains an edge $e$ and $\{e\}$ is the desired matching.
Assume next that $k\geq 2$.
Suppose $H-Z$ contains no edges.
This implies that $H$ is bipartite with bipartition $(V(H)\setminus Z,Z)$.
By K\"onig's Theorem, the matching number of $H$ equals the vertex cover number and hence $H$ contains a matching $N$ of size $2k-1$.
Let $M$ be the matching obtained from $N$ by deleting one of the matching edges $x_ip$ and $y_iq$ if both exist in $N$.
Therefore, $|M|= k$ and $M$ is the desired matching.

In the following we may assume that $H-Z$ contains edges.
Suppose there is a vertex in $Z$ which is an isolated vertex in $H$, by symmetry, say $x_k$.
Let $e=uy_k$ be an edge of $H$ incident to $y_k$ if such an edge exists otherwise let $e=uv$ be some edge in $H-Z$.
As $x_k$ is an isolated vertex, 
$\tau(H-\{u,x_k,y_k\})\geq 2k-3$ if $e=uy_k$ and
$\tau(H-\{u,v,x_k,y_k\})\geq 2k-3$ if $e=uv$, since $y_k$ is also isolated.
By the induction hypothesis, 
there exists a matching $M'=\{m_1,\ldots,m_{k-1}\}$ such that $m_i\cap \{x_j,y_j\}=\emptyset$ for $i\neq j$.
Thus $M'\cup \{e\}$ is the desired matching.

Therefore, we may assume that every vertex in $Z$ has a neighbor.
We obtain a desired matching by induction if $\tau(H)\geq 2k$ by deleting $x_k$, $y_k$, and a neighbor of $x_k$.
Thus we may assume $\tau(H)=2k-1$.

Let $H'$ be the induced subgraph of $H$ which is obtained from $H$ by deleting all isolated vertices.
We may assume that the $\tau(H'-e)<\tau(H')$ for every $e\in E(H')$ and $\tau(H'-r)<\tau(H')$ for every $r\in V(H')\setminus Z$. 
Moreover, if $\tau(H'-x_i)=\tau(H')$,
then let $r$ be a neighbor of $y_i$ and the statement follows by induction because $\tau(H'-\{x_i,y_i,r\})\geq \tau(H')-2$.
This implies that $H'$ is a $\tau$-critical graph.

Since $Z$ is an independent set, 
complement of an independent set is a vertex cover, 
and $\tau(H')=2k-1$, 
we conclude that $|V(H')|\geq 2k-1 + 2k=4k-1$.
Thus $\tau(H')< |V(H')|/2$, which contradicts the theorem of Erd\H{o}s and Gallai mentioned before.
\end{proof}

\section{Odd cycles through prescribed vertices}

In this section we present several results concerning the Erd\H{o}s-P\'osa property of odd $S$-cycles in highly connected graphs.
This extends the results concerning the Erd\H{o}s-P\'osa property of odd cycles in highly connected graphs.
Furthermore, assuming a slightly higher connectivity,
we show how known results follow easily from Theorem \ref{parity linked fixed set}.

\begin{lemma}\label{lemma erdos posa}
Let $k\in \mathbb{N}$ and let $G$ be a $50k$-connected graph.
Let $S$ be a set of $k$ vertices
and suppose there is a nice partition $(A, B)$ of $G$
with a matching $M$  of size $k$ in $G_{A,B}$ such that every edge in $M$ covers at most one vertex of $S$.
Then $G$ contains $k$ disjoint odd $S$-cycles.
%\begin{enumerate}
%	\item $G$ contains $k$ disjoint odd $S$-cycles.
%	\item For all nice partitions $(A, B)$ of $G$, there is no matching $M$ of size $k$ in $G_{A,B}$
%	where an edge in $M$ covers at most  one vertex of $S$.
%\end{enumerate}
\end{lemma}
Observe that every set of $k$ disjoint odd $S$ cycles
lead in any partition $(A,B)$ of $G$ to a matching $M$ of size $k$ in $G_{A,B}$ such that every edge in $M$ covers at most one vertex of $S$ (if $|S|=k$),
because every odd cycle in $G$ uses at least one edge in $G_{A,B}$.
\begin{proof}[Proof of Lemma~\ref{lemma erdos posa}]
%Suppose first that $G$ contains $k$ disjoint $S$-cycles and 
%let $(A, B)$ be a partition of $G$. 
%Since every odd cycle contains at least one edge of $E(G_{A,B})$ and by picking exactly one of these edges from every odd $S$-cycle,
%we get a matching of size $k$ in $G_{A,B}$
%where one edge does not cover two vertices of $S$, 
%which implies that the second statement does not hold.

%let $(A,B)$ be a nice partition
%such that there is a matching $M$ 
% of size $k$ in $G_{A,B}$
%and one edge in $M$ does not cover two vertices of $S$.

Let $T=\{t_1,\ldots,t_k\}$ be a set of vertices distinct from $S=\{s_1,\ldots,s_k\}$ and distinct from the vertices covered by $M$.
By Theorem \ref{parity linked fixed set}, 
there are disjoint paths $P_1,\ldots,P_k$ with prescribed parity such that $P_i$ connects $s_i$ and $t_i$,
because $M$ is a parity breaking matching for $\{\{s_1,t_1\},\ldots,\{s_k,t_k\}\}$.

%Let $(\tilde{A},\tilde{B})$ be any partition of $G$.
%By suitably choosing the parity of $P_i$, every path uses at least one edge in $G_{\tilde{A},\tilde{B}}$ and 
%hence there is a matching in $G_{\tilde{A},\tilde{B}}$ of size at least $k$ such that each edge covers at most one vertex of $S$.

For every $s_i$, add to $G$ a vertex $s_i'$ such that $N(s_i)=N(s_i')$ and denote this new graph by $G'$.
Note that $G'$ is $50k$-connected.
Let $(A', B')$ be a nice partition of $G'$.
Let $Q_1,\ldots,Q_k$ be disjoint paths in $G$ such that $Q_i$ connects $s_i$ and $t_i$
and the parity of the length of $Q_i$ is odd if and only if $s_i$ and $t_i$ belong both to $A'$ or $B'$.
For every $i$, the path $Q_i$ uses at least one edge in $G_{A',B'}$.
Let $N$ be a collection of $k$ edges $e_1,\ldots,e_k$ in $G_{A',B'}$ such that $e_i\in E(Q_i)$.
Clearly, $N$ does not cover a vertex of the set $\{s_1',\ldots,s_k'\}$.
Thus $N$ is a parity breaking matching for $\{\{s_1,s_1'\},\ldots,\{s_k,s_k'\}\}$ in $G'$.
By Theorem \ref{parity linked fixed set}, there are disjoint paths $P_1',\ldots,P_k'$ of odd length
where $P_i'$ joins $s_i$ and $s_i'$.

Since $N(s_i)=N(s_i')$,
this in turn implies the existence of $k$ disjoint odd $S$-cycles $C_1,\ldots, C_k$ in $G$ where $C_i$ contains $s_i$. 
\end{proof}

After having proved Lemma \ref{lemma erdos posa},
it is not difficult to prove the Erd\H{o}s-P\'osa property of odd $S$-cycles in highly connected graphs.

\begin{theorem}\label{S-cycles}
Let $k\in \mathbb{N}$ and let $G$ be a $50k$-connected graph.
Let $S$ be a set of vertices.
At least one of the following statements holds.
\begin{enumerate}
	\item $G$ contains $k$ disjoint odd $S$-cycles.
	\item There is a set $X$ with $|X|=2k-2$ such that $G-X$ does not contain an $S$-cycle.
\end{enumerate}
\end{theorem}
\begin{proof}
We may assume that $|S|\geq 2k-1$, otherwise the statement is trivial.
Let $X$ be a minimum odd cycle cover of $G$. 
We may assume that $|X|\geq 2k-1$.
Let $(A, B)$ be nice partition of $G$ induced by $X$.
Thus $\tau(G_{A,B})\geq 2k-1$, by \eqref{vertex cover occ}.
By \eqref{vertex cover matching},
the graph $G_{A,B}$ contains a matching $M$ of size $k$.
Let $S'\subseteq S$ be a set of $k$ vertices such that no edge in $M$ covers two vertices in $S'$.
Since an edge covers at most two vertices, $S'$ exists.
Using Lemma \ref{lemma erdos posa}, 
there are $k$ disjoint odd $S'$-cycles in $G$
and hence $k$ disjoint odd $S$-cycles in $G$.
\end{proof}

If a graph $G$ is $50k$-connected and $G$ does not contain $k$ disjoint odd cycles because of the trivial reason that $|S|\leq k-1$,
then $G$ is even almost bipartite.

\begin{theorem}\label{S-cycles vertex}
Let $k\in \mathbb{N}$ and let $G$ be a $50k$-connected graph.
Let $S$ be a set of at least $k$ vertices.
At least one of the following statements holds.
\begin{enumerate}
	\item $G$ contains $k$ disjoint odd $S$-cycles.
	\item There is a set $X$ with $|X|=\min_{S'\subseteq S, |S'|=k}\{2k-2+\tau(G[S'])\}$ such that $G-X$ is bipartite.
\end{enumerate}
\end{theorem}
\begin{proof}
Let $X$ be a minimum odd cycle cover
and let $S'\subseteq S$ such that $|S'|=k$ and $\tau(G[S'])=\min_{S'\subseteq S, |S''|=k} \tau(G[S''])$.
We may assume that $|X|\geq 2k-1+\tau(G[S'])$.
Let $(A, B)$ be nice partition of $G$ and 
let $Y$ be a minimum vertex cover of $G[S']$.
Since $X$ is a minimum vertex cover of $G_{A,B}$ by \eqref{vertex cover occ},
we conclude $\tau(G_{A,B}-Y)\geq 2k-1$.
Note that $S'-Y$ is an independent set in $G_{A,B}-Y$.
Thus by using (\ref{vertex cover matching}), 
this in turn implies the existence of a matching $M$ of size $k$ in $G_{A,B}-Y$ 
such that every edge in $M$ covers at most one vertex of $S'$.
Using Lemma \ref{lemma erdos posa}, this implies the existence of $k$ disjoint odd $S'$-cycles in $G$.
\end{proof}

Let $\tau_k(G[S])=\min_{S'\subseteq S, |S''|=k} \tau(G[S''])$.
Note that $\tau_k(G[S])\leq k-1$ and thus $2k-2+\tau_k(G[S]) \leq 3k-3$.
Moreover, the bound ``$2k-2+\tau_k(G[S])$'' is sharp for every possible value of $\tau_k(G[S])$ no matter how large the
connectivity of $G$ is.
To see this, let $G$ arise from a large complete bipartite graph with bipartition $(A,B)$ by
adding the edges of a clique on $2k-1$ vertices to $A$ and the edges of a clique on $\tau$ vertices to $B$ for some $1\leq\tau\leq k$.
Let $S$ be a set of $k$ vertices in $B$ containing the $\tau$-clique.
Hence $\tau(G[S])=\tau-1$, there do not exist $k$ disjoint odd $S$-cycles, 
and there is no set $X$ of $2k-3+\tau(G[S])$ vertices such that $G-X$ is bipartite.

We proceed with a proof of Theorem~\ref{EPoddScycles}.
\begin{proof}[Proof of Theorem~\ref{EPoddScycles}]
Theorem~\ref{S-cycles} proves the first part of the statement
and if $|S|\geq k$,
then the observation $2k-2+\tau_k(G[S]) \leq 3k-3$ together with Theorem~\ref{S-cycles vertex} proves the second part of the statement.
\end{proof}

We conclude the paper with two results about disjoint odd cycles; that is, odd $S$-cycles with $S=V(G)$.

\begin{corollary}
Let $k\in \mathbb{N}$ and let $G$ be a $50k$-connected graph.
Exactly one of the following statements holds.
\begin{enumerate}
	\item $G$ contains $k$ disjoint odd cycles.
	\item For every nice partition $(A, B)$, the graph $G_{A,B}$ does not contain a matching of size~$k$.
\end{enumerate}
\end{corollary}

\begin{proof}
If the first statement holds, then the second does clearly not hold.

Suppose that the second statement does not hold.
Let $(A,B)$ be a nice partition of $G$
and let $M$ be a matching of size $k$ in $G_{A,B}$.
If $G$ contains an independent set $I$ of size $k$,
then every edge in $M$ covers at most one vertex in $I$.
By Theorem~\ref{S-cycles vertex}, 
the graph $G$ contains $k$ disjoint odd $I$-cycles.
Thus we may assume that $G$ does not contain an independent set of size $k$.
We claim that $G$ contains $k$ disjoint triangles.
This can be seen as follows.
Select a vertex $u$ in $G$.
As its neighborhood does not contain an independent set of size $k$,
it contains an edge $vw$.
Delete the triangle $uvwu$ from $G$ and iterate this process $k$ times.
As $G$ is $50k$-connected,
the graph after the $i$-th iteration is still $(50k-3i)$-connected.
\end{proof}

The following corollary is already proven by Thomassen \cite{Tho01} and Rautenbach and Reed \cite{RR01} with a higher connectivity bound.
Later Kawarabayashi and Reed \cite{KR09} and Kawarabayashi and Wollan \cite{KW06} improved this bound to $24k$ and $\frac{31}{2}k$, respectively.

\begin{corollary}
Let $k\in \mathbb{N}$ and let $G$ be a $50k$-connected graph.
At least one of the following statements holds.
\begin{enumerate}
	\item $G$ contains $k$ disjoint odd cycles.
	\item $G$ contains a set $X$ of $2k-2$ vertices such that $G-X$ is bipartite.
\end{enumerate}
\end{corollary}

\begin{proof}
Suppose the second statement does not hold.
Let $(A, B)$ be a nice partition of $G$ induced by a minimum odd cycle cover $X$.
By \eqref{vertex cover occ}, $X$ is a minimum vertex cover of $G_{A,B}$.
Since $|X|\geq 2k-1$ by our assumption,
(\ref{vertex cover matching}) implies that
the graph $G_{A,B}$ contains a matching of size $k$.
Because this holds for every minimum odd cycle cover and so for every nice partition of $G$, 
the statement follows from the previous corollary.
\end{proof}

Clearly, assuming $50k$-connectivity in our results is not the best bound in terms of $k$ one can hope for.
However, it is essentially best possible in the sense that as one can easily construct graphs
that show that linear connectivity in $k$ is necessary.
It would be interesting to know which connectivity is needed to ensure that our results hold.
It even seems possible that the approach via a parity-$k$-linkage theorem cannot lead to the best connectivity bound.

\bibliographystyle{amsplain}
\bibliography{erdosposa}

\vfill

\small
\vskip2mm plus 1fill
\noindent
Version \today{}
\bigbreak

\noindent
Felix Joos
{\tt <felix.joos@uni-ulm.de>}\\
Institut f\"ur Optimierung und Operations Research\\
Universit\"at Ulm, Ulm\\
Germany

\end{document}